\documentclass[12pt]{article}

\usepackage{amssymb,amsfonts,amsmath,txfonts,amscd,makeidx,url}
\usepackage{latexsym}
\usepackage{pdfsync}

\newtheorem{theorem}{Theorem}
\newtheorem{lemma}[theorem]{Lemma}

\newtheorem{proposition}[theorem]{Proposition}
\newtheorem{corollary}[theorem]{Corollary}

\newtheorem{conjecture}{Conjecture}

\newenvironment{proof}{\paragraph{\it Proof.}}{$\square$\vskip0.4cm}

\newcommand{\beq}{\begin{eqnarray}}
\newcommand{\eeq}{\end{eqnarray}}
\newcommand{\bes}{\begin{eqnarray*}}
\newcommand{\ees}{\end{eqnarray*}}

\newcommand{\gl}{{\mathfrak g}\mathfrak{l}}
\newcommand{\GL}{{\rm GL}}
\newcommand{\g}{{\mathfrak g}}

\newcommand{\F}{{\mathbb F}}
\newcommand{\vv}{{\mathbf v}}
\newcommand{\1}{{\mathbf 1}}

\renewcommand{\char}{{\rm char}}
\renewcommand{\P}{{\mathcal P}}
\newcommand{\w}{{\mathbf w}}

\newcommand{\calM}{{\cal M}}
\newcommand{\calP}{{\cal P}}

\newcommand{\N}{{\mathbb N}}

\newcommand{\calC}{{\cal C}}

\newcommand{\calV}{{I}}
\newcommand{\calE}{{E}}

\newcommand{\C}{{\mathbb C}}

\newcommand{\Z}{{\mathbb Z}}
\newcommand{\X}{{\mathcal X}}
\newcommand{\Y}{{\mathcal Y}}

\newcommand{\K}{\mathbb K}

\newcommand{\V}{{\mathbb V}}
\newcommand{\M}{{\mathbb M}}

\newcommand{\rmG}{{\rm G}}

\newcommand{\0}{{\mathbf 0}}

\newcommand{\G}{{\rm G}}

\newcommand{\im}{\mathop{\rm im}\nolimits}

\newcommand{\tr}{\mathop{\rm tr}\nolimits}

\newcommand{\Hom}{\mathop{\rm Hom}\nolimits}

\begin{document}
\title{Kac's conjecture from  Nakajima quiver varieties}
\author{
 Tam\'as Hausel
\\ {\it University of Oxford}
\\{\tt hausel@maths.ox.ac.uk}}
\maketitle

\begin{abstract} We prove a generating function formula for the Betti numbers of  Nakajima quiver varieties. We prove that  it is a  $q$-deformation of the Weyl-Kac character formula. In particular this implies  that the constant term of the polynomial counting the number of absolutely indecomposable representations of a quiver equals the multiplicity of a a certain weight in the corresponding Kac-Moody algebra, which was conjectured by Kac in 1982. 
\end{abstract}

Let $\Gamma=(\calV,\calE)$ be a quiver that is an oriented 
graph on a finite set $\calV=\{1,\dots,n\}$ with $\calE\subset \calV\times \calV$ a finite multiset  of oriented
(perhaps multiple but no loop) edges. Given two dimension vectors $\vv=(v_i)\in \N^\calV$ and $\w=(w_i)\in \N^\calV$  Nakajima \cite{nakajima0,nakajima} constructs, as holomorphic symplectic quotient, a complex variety $\calM(\vv,\w)$ of dimension  $2d_{\vv,\w}$, which we call a {\em Nakajima quiver variety}. In \cite{nakajima-tq} Nakajima found a combinatorial algorithm to determine the Betti numbers of these varieties. 

Here we prove and study the following generating function of  Betti numbers of  Nakajima quiver varieties.

\begin{theorem} Fix $\w\in \N^\calV$. Denote $b_i(\calM(\vv,\w)):=\dim\left(H^i(\calM(\vv,\w))\right)$. Then in the notation of \eqref{dimension},\eqref{formal}, \eqref{partitionotation}
\beq \label{hausel} \sum_{\vv\in \N^\calV} \sum_{i=0}^{{d_{\vv,\w}}} b_{2i}(\calM(\vv,\w)) q^{d_{\vv,\w}-i} X^{\vv}= \frac{\displaystyle{\sum_{\lambda  \in \calP^\calV} }  \frac{\left(\prod_{(i,j)\in \calE} q^{\langle \lambda^i,\lambda^j\rangle} \right)\left(\prod_{i\in \calV} q^{\langle\lambda^i,1^{\w_i}\rangle}\right)}{\prod_{i\in \calV} \left(q^{\langle\lambda^i,\lambda^i\rangle}\prod_k \prod_{j=1}^{m_k(\lambda^i)} (1-q^{-j})\right) }X^{|\lambda|}}{\displaystyle{\sum_{\lambda  \in \calP^\calV}}   \frac{\prod_{(i,j)\in \calE} q^{\langle \lambda^i,\lambda^j\rangle}}{\prod_{i\in \calV} \left(q^{\langle\lambda^i,\lambda^i\rangle}\prod_k \prod_{j=1}^{m_k(\lambda^i)} (1-q^{-j})\right) }X^{|\lambda|}}
\eeq  \label{maint}
\end{theorem}

The method of the proof is, as announced in \cite{hausel-betti}, {\em arithmetic Fourier transform} for holomorphically symplectic quotients. It proceeds by counting points of $\calM(\vv,\w)$ over finite fields $\F_q$ using a Fourier transform technique, then showing that the mixed Hodge structure on $H^*(\calM(\vv,\w))$ is pure and concludes by applying Katz's result \cite[Appendix, Theorem 6.2]{hausel-villegas}
connecting the arithmetic and the cohomology  of so-called { polynomial-count}  varieties.

The other theme of this paper is a combinatorial study of the formula \eqref{hausel}. Using the main result of \cite{nakajima} we prove that
the Weyl-Kac character formula is the $q=0$ specialization of \eqref{hausel}. We then interpret Hua's formula 
\cite[Theorem 4.9]{hua} as a $q$-deformation of the Kac denominator formula for the denominator on the RHS of \eqref{hausel}. This in particular implies for $\alpha\in \N^\calV$ that the constant term of Kac's $A$-polynomial counting absolutely indecomposable representations of a quiver of dimension $\alpha$ over a finite field agrees with the multiplicity of $\alpha$ in the Kac-Moody algebra $\g(\Gamma)$ associated to the quiver. This completes\footnote{After this result and Theorem~\ref{maint} were announced in \cite{hausel-betti} Mozgovoy in  \cite{mozgovoy2} found an alternative proof of
Theorem~\ref{maint} by combining \cite{crawley-boevey-etal} and \cite{hua}. Our proof is independent of \cite{crawley-boevey-etal} and gives a more general approach
to calculate the Betti numbers of holomorphic symplectic quotients of vector spaces by linear actions of reductive groups.} a proof of Kac's \cite[Conjecture 1]{kac2}. This conjecture was known for finite and tame quivers \cite{kac2}, and for all quivers  with $\alpha$ indivisible \cite{crawley-boevey-etal} by study of similar quiver varieties, but was not fully known for any single wild quiver.  

\begin{paragraph} {\bf Acknowledgment.}  I am grateful for Hiraku Nakajima for suggesting the  possibility of the application of \eqref{hausel}  for Kac's conjecture,  for Fernando Rodriguez-Villegas for drawing my attention to Hua's work, 
for Nick Proudfoot for introducing me to Kac's conjecture and an anonymous referee for useful comments. My research into the Betti numbers of Nakajima quiver varieties started during the "Geometry of Quiver varieties"  seminar \cite{quiver-seminar} at UC Berkeley in Fall 2000, where I learned a great deal from the speakers of the seminar. This work has been supported by a Royal Society University Research Fellowship, NSF grants DMS-0305505 and  DMS-0604775 and an Alfred Sloan Fellowship 2005-2007.

\end{paragraph} 

\section{Arithmetic Fourier transform}

\paragraph{Fourier transform over finite vector spaces.} Let $\K$ be the finite field $\F_q$ on $q=p^r$ elements, where $p$ is a prime. 
 Let $V$ be a finite dimensional vector space over $\K$. Consider the   finite abelian group underlying $V$. The set of characters $$\widehat{V}:=\{ \chi:V\to \C^\times\ | \ \  \chi(v_1-v_2)=\chi(v_1)\chi(v_2)^{-1} \mbox{ for all } v,w\in V\} $$ forms a $\K$-vector space by the definitions 
$\chi-\psi=\chi\psi^{-1}$,  and $\alpha\chi(v)=\chi(\alpha v)$, where $\alpha\in \K$ and $v\in V$.  We have $|V|=|\widehat{V}|$, where $|S|$ denotes the cardinality of any finite set $S$. 
  The  orthogonality relations on characters of finite abelian groups say that 
 $$\sum_{v\in V} \chi_1(v)\chi_2(v)=\left\{ \begin{array}{cc} |V| & \chi_1=-\chi_2\\ 0 & \mbox {ow} \end{array} \right.$$
 and dually  $$\sum_{\chi\in \widehat{V}} \chi(v_1)\chi(v_2)=\left\{ \begin{array}{cc} |\widehat{V}| & v_1=-v_2\\ 0 & \mbox {ow} \end{array} \right.$$

The Fourier transform of a function  $f:V\to\C $  is  defined to be the function  $\widehat{f}:\widehat{V}\to \C$
 $$\widehat{f}(\chi)=|V|^{-1/2} \sum_{v\in V} f(v)\chi(v).$$ With this definition we find that for any $\chi_1,\chi_2\in \widehat{V}$ 
 $$\widehat{\chi}_1(\chi_2)=|V|^{-1/2} \sum_{v\in V} \chi_1(v)\chi_2(v)=|V|^{1/2} \delta_{-\chi_1}(\chi_2),$$ where for any element of any finite abelian
 group $w\in W$, we denote the delta function $\delta_w:W\to \C$, given by $\delta_w(v)=0$ unless $v=w$ when $\delta_w(w)=1$. Thus for any $\chi\in \widehat{V}$ we have \beq \label{fourierchar}\widehat{\chi}=|V|^{1/2} \delta_{-\chi}.\eeq

We can think of an element $v\in V$ as a character $v:\widehat{V}\to \C^\times$ given by $v(\chi):=\chi(v)$, this way we can identify 
 $\widehat{\widehat{V}}=V$. Thus the Fourier transform of a function $g:\widehat{V}\to \C$ becomes $\widehat{g}:V\to \C$
 with the formula $$\widehat{g}(v)=|V|^{-1/2} \sum_{\chi\in \widehat{V}}g(\chi) \chi(v).$$ For any $f:V\to \C$ and $x\in V$ we have the {\em Fourier Inversion Formula}
 \beq \label{fif} \widehat{\!\!\widehat{f}}(x)=|V|^{-1} \sum_{v\in V} \sum_{\chi \in \widehat{V}} f(v) \chi(x)\chi(v)=  |V|^{-1}  \sum_{v\in V} f(v) |V| \delta_{v+x}(0) = f(-x).\eeq

 In order to identify $\widehat{V}$ with $V^*$ we
fix $\Psi:\K\to \C^\times$ a non-trivial additive character in the remainder of the paper. 
 For  $w\in V^*$  we consider the  character
$\chi_w:V\to \C^\times$  which at $v\in V$ is given by $\chi_w(v):=\Psi(w(v))$. 
  Consider the map $\pi_w: V\to \K$ given by $\pi_w(v)=w(v)$. When $w\neq 0$ the linear
 map $\pi_w$ is surjective and so   
 $$\sum_{v\in V} \chi_w(v)=\sum_{x\in \K} \sum_{\pi_w(v)=x} \Psi(x) = \sum_{x\in \K} |\ker\pi_w| \Psi(x)=0,$$ as $\Psi$ is a non-trivial character of $\K$; therefore $\chi_w\neq \chi_0$ and so
 $\chi_{w_1}=\chi_{w_2}$ if and only if $w_1=w_2$.  Thus $w\mapsto \chi_w$ indentifies \beq \label{ident} V^*=\{\chi_w\}_{w\in V^*}=\widehat{V}\eeq because  $|V^*|=|\widehat{V}|=|V|$.

\paragraph{Counting solutions of moment map equations over finite fields.}  Let 
$\rmG$ be an algebraic group over $\K$, $\g$ its Lie algebra. Consider a representation $\rho: \rmG\to \GL(\V)$ of $\rmG$ 
on a $\K$-vector space $\V$, inducing the Lie algebra representation 
$\varrho: \g \to \gl(\V)$. 
This induces an action $\overline{\rho}: \rmG\to \GL(\M)$ and $\overline{\varrho}:\g\to \gl(\M)$ on $\M=\V\times \V^*$. 
 The vector
space $\M$ has a natural symplectic structure; defined by the natural pairing $\langle v,w\rangle =w(v)$, 
with $v\in \V$
and $w\in \V^*$ by the formula: $$\omega((v_1,w_1),(v_2,w_2))=w_2(v_1)-w_2(v_1).$$ 
With respect to this symplectic form a moment map 
$$\mu: \V\times \V^* \to \g^*$$ 
of $\rho$ is given at $x\in \g$ by  \beq \label{momentmap} 
\langle \mu(v,w),x\rangle=\langle \varrho(x)v,w\rangle.
\eeq 
This $\mu$ is indeed a moment map because if for an $x\in \g$ we introduce $f_x:\M\to \K$ by \beq \label{fx}f^x(v_1,w_1)=\langle \mu (v_1,w_1),x\rangle\eeq and take the constant vector field  $Y$ on $\M$ given by $(v_2,w_2)\in \M$ then \beq \label{momentdefn} \nonumber df^x_{(v_1,w_1)}(Y)&=&Yf^x(v_1,w_1)=Y\langle\varrho(x)v_1,w_1\rangle=\\ &=& \langle\varrho(x)v_2,w_1\rangle+\langle\varrho(x)v_1,w_2\rangle=-\langle v_2,\varrho^*(x)w_1\rangle+\langle\varrho(x)v_1,w_2\rangle=\omega(\overline{\varrho}(x)(v_1,w_1),(v_2,w_2)),\eeq which is the defining property of a moment map. 

Our main proposition determines  $$\#_\mu(\xi):=\sum_{(v,w)\in M} \delta_{\mu(v,w)}(\xi)$$ the number of solutions $(v,w)\in \V\times \V^*$  of the moment map equation $\mu(v,w)=\xi$ for a fixed $\xi\in \g^*$. 
The main observation is that the Fourier transform of $\#_\mu:\g^*\to \C$ considered as an  $\N\subset \C$-valued function on $\g^*$ is
a simple function on $\g$.  Indeed  for $x\in \g$ we find
\begin{eqnarray} \nonumber \widehat{\#}_\mu(x)&=&\sum_{(v,w)\in \M}\widehat{\delta}_{\mu(v,w)}(x)=|\g|^{1/2} \sum_{(v,w)\in \M} \chi_{\mu(v,w)}(-x)=|\g|^{1/2} \sum_{(v,w)\in \M} \Psi(\langle \mu(v,w),-x\rangle)\\ &= & \nonumber |\g|^{1/2} \sum_{(v,w)\in \M} \Psi(\langle \varrho(-x)v,w\rangle)=|\g|^{1/2}\sum_{v\in \V} \sum_{w\in \V^*} \chi_w(\rho(-x)v)=|\g|^{1/2}|\V| \sum_{v\in \V} \delta_0(\rho(-x)v)\\ &=&   |\g|^{1/2}|\V| a_\rho(-x) \label{fouriersolution} \end{eqnarray}
where we used the notation 
$a_\varrho:\g\to {\mathbb N}\subset \C$ given at $x\in \g$ by 
\beq \label{kernel} a_\varrho(x):= \sum_{v\in \V} \delta_0(\rho(x)v) =|\ker(\varrho(x))|. \eeq In particular $a_\varrho(X)$ is always a power of $q$.  Taking Fourier transforms of both sides of \eqref{fouriersolution}
and using \eqref{fif} we get our main proposition:
\begin{proposition} The number of solutions of the equation  $\mu(v,w)=\xi$
 over the finite field $\F_q$ equals:  \bes \# \{(v,w)\in \M\,\, |\,\, \mu(v,w)= \xi\}=\#_\mu(\xi) = |\g|^{-1/2} |\V| \widehat{a}_{\varrho
}(\xi)= |\g|^{-1} |\V| \sum_{x\in \g} a_{\varrho
}(x) \Psi(\langle x, \xi\rangle)\ees
\label{main}
\end{proposition}

\section{Nakajima quiver varieties} 
These varieties were introduced and studied in \cite{nakajima0,nakajima}, here we go through the definitions and for the sake of self-containedness we reprove some results we need later. 

Let $\Gamma=(\calV,\calE)$ be a quiver, i.e. an oriented 
graph on a finite set $\calV=\{1,\dots,n\}$ with $\calE\subset \calV\times \calV$ a finite multiset of oriented
(perhaps multiple and loop) edges. We will think of $\calE$ as an abstract finite set together with source and target maps
$s,t:E\to I$ so that the oriented edge $e=(s(e),t(e))\in \calV\times \calV$.   To each vertex $i$ of the graph we associate two finite dimensional $\K$ vector
spaces $V_i$ and $W_i$. We call $(\vv_1,\dots,\vv_n,\w_1,\dots, \w_n)=(\vv,\w)$ the dimension vector, where $\vv_i=\dim(V_i)$ and $\w_i=\dim(W_i)$. To this data we associate the grand vector space:
$$\V_{\vv,\w}=\bigoplus_{e\in \calE} \Hom(V_{s(e)},V_{t(e)}) \oplus \bigoplus_{i\in \calV} \Hom (W_i,V_i),$$
the group and its Lie algebra $$\rmG_{\vv}=\varprod_{i\in \calV} \GL(V_i)$$ $$ \g_{\vv}=\bigoplus_{i\in \calV} \gl(V_i),$$ and the natural representation 
$$\rho_{\vv,\w}:\rmG_{\vv}\to \GL(\V_{\vv,\w}),$$ with derivative $$\varrho_{\vv,\w}:\g_{\vv}\to \gl(\V_{\vv,\w}).$$ The action is  from both left and right
on the first term, and from the left on the second.  We will also need the double representation $$\overline{\rho}_{\vv,\w}:=\rho_{\vv,\w}\times \rho^*_{\vv,\w}:\rmG_{\vv}\to \GL(\V_{\vv,\w}\times \V^*_{\vv,\w}),$$
with derivative: $$\overline{\varrho}_{\vv,\w}:=\varrho_{\vv,\w}\oplus \varrho^*_{\vv,\w}:\g_{\vv}\to \gl(\V_{\vv,\w}\times \V^*_{\vv,\w}).$$  Concretely, if $(v,w)\in \V_{\vv,\w}\times \V^*_{\vv,\w}$ is given by \beq \label{concret} (A_{e},I_i,B_{e}, J_i)_{e\in \calE, i\in \calV, }, \eeq where 
$A_{e}\in \Hom(V_{s(e)},V_{t(e)})$, $B_{e} \in \Hom(V_{t(e)},V_{s(e)})$ for $e\in \calE$
, $I_i\in \Hom(W_i,V_i)$  and $J_i\in \Hom(V_i,W_i)$ for $i\in\calV$ and $g=(g_i)_{i\in\calV}\in \g_v$, where $g_i\in \GL(V_i)$  and finally $x=(x_i)_{i\in\calV}\in \g_v$, where $x_i\in \gl(V_i)$ then $$\overline{\rho}_{\vv,\w}(g)(v,w)=({\rho}_{\vv,\w}(g)(v),{\rho}^*_{\vv,\w}(g)(w))=(g_{t(e)}A_{e}g_{s(e)}^{-1},g_iI_i,g_{s(e)}B_{e}g_{t(e)}^{-1},Jg_i^{-1})_{e\in \calE, i\in \calV, },$$ and \beq \label{varrhobar}\overline{\varrho}_{\vv,\w}(x)(v,w)=({\varrho}_{\vv,\w}(x)(v),{\varrho}_{\vv,\w}^*(x)(w))=(x_{t(e)}A_e-A_{e}x_{s(e)},x_iI_i,x_{s(e)}B_{e}-B_{e}x_{t(e)},-Jx_i)_{e\in \calE, i\in \calV}.\eeq

We now have $\rmG_{\vv}$ acting on $\M_{\vv,\w}=\V_{\vv,\w}\times \V_{\vv,\w}^*$ preserving
the symplectic form  with moment map $$\mu_{\vv,\w}:\V_{\vv,\w}\times\V_{\vv,\w}^*\to \g_{\vv}^*$$ given by \beq \label{quivermoment} \langle \mu_{\vv,\w} (v,w),x\rangle=\langle \varrho_{\vv,\w} (x)v,w \rangle. \eeq   Concretely if $(v,w)\in \V_{\vv,\w}\times\V_{\vv,\w}$ is given by \eqref{concret}, then we can choose \beq \label{explicit}\mu_{\vv,\w}(v,w)=\left(I_iJ_i+\sum_{e\in s^{-1}(i)} B_{e}A_e-\sum_{e\in t^{-1}(i)} A_{e}B_{e}\right)^*_{i\in\calV}\in\g^*_\vv,\eeq where we identify any element $x=(x_i)_{i\in \calV}\in \g_\vv$ with the linear form $x^*:\g_\vv\to \K$ given by $$x^*\left((y_i)_{i\in \calV}\right)= \sum_{i\in \calV}\tr (x_i y_i)\in \K.$$

\begin{lemma} \label{lemma} Let $\w=\0$ and $\vv\in \N^\calV$ arbitrary. For $\lambda\in \K^\times$ let 
${\lambda \1_\vv}=(\lambda Id_{V_1},\dots,\lambda Id_{V_n})\in \g_\vv$.  Further assume ${\rm char} (\K)\nmid \sum_{i=1}\vv_i$ or equivalently $0_\K\neq 1_\K(\sum_{i=1}\vv_i)$. Then the equation $\mu_{\vv,\0}(v,w)=\lambda\1^*_\vv$ has a solution if and only if $\vv=\0$. \end{lemma}
\begin{proof} As $\w=\0$ by \eqref{varrhobar} we have ${\varrho}_{\vv,\0}(\1_\vv)(v)=0$. Now assume $\mu_{\vv,\0}(v,w)=\lambda\1^*_\vv$ then by \eqref{quivermoment} we have $$ \lambda\sum_{i\in \calV}\vv_i=\langle \mu_{\vv,\0} (v,w),\1_v\rangle=\langle \varrho_{\vv,\0} (\1_\vv)v,w\rangle=0,$$ thus ${\rm char}(\K)\mid \sum_{i\in \calV} \vv_i$.
\end{proof}

\begin{proposition}\label{crucialcor} Let $\w,\vv\in \N^\calV$ arbitrary and $\char(\K)=0$ or $\char(\K)>\sum_{i\in \calV} \vv_i.$ For $\lambda\in \K^\times$, and a solution \beq \label{solnlambda} \mu_{\vv,\w}(v,w)=\lambda\1^*_\vv\eeq assume that for a $x\in \g_\vv$  \beq \label{fixed}\overline{\varrho}_{\vv,\w}(x)(v,w)=0\eeq then $x=0$. 
\end{proposition}
\begin{proof}Let $V^\prime_i:=\im(x_i)\leq V_i$ with dimension $\vv^\prime_i$. As $\overline{\varrho}_{\vv,\w}(x)(v,w)=0$ we have that $J_ix_i=0$ for all $i\in \calV$ and so  $J_i$ is zero on $V^\prime_i\leq V_i$. On the other hand we also have that $x_{t(e)}A_e-A_ex_{s(e)}=0$ which implies that $A_e$ maps 
$V^\prime_{s(e)}\leq V_{s(e)}$ to $V^\prime_{t(e)}\leq V_{t(e)}$  and 
similarly $B_{e}$ maps $V^\prime_{t(e)}\leq V_{t(e)}$ to $V^\prime_{s(e)}\leq V_{s(e)}$. Finally $ \lambda\1^*_\vv$ restricts to $\lambda\1^*_{\vv^\prime}$. Thus a solution of \eqref{solnlambda} reduces to a solution of 
$\mu_{\vv^\prime,\0}(v^\prime,w^\prime)=\lambda\1^*_{\vv^\prime}$. By our assumption on the characteristic of $\K$ and 
Lemma~\ref{lemma} we have that $\sum_{i\in\calV} \vv^\prime_i =0$ or equivalently $x=0$. 
\end{proof}

\begin{corollary}\label{simple} Let $\w,\vv\in \N^\calV$ arbitrary and $\char(\K)=0$ or $\char(\K)>\sum_{i\in \calV} \vv_i.$ 
Assume that a solution of \eqref{solnlambda} is invariant under $g\in \rmG_\vv$, then $g=\1_\vv$. \end{corollary}
\begin{proof} We have $\overline{\rho}_{\vv,\w}(g)(v,w)=(v,w)$ and so $g-\1_\vv\in \g_\vv$ satisfies $\overline{\varrho}_{\vv,\w}(g-\1_\vv)(v,w)=0$ thus Proposition~\ref{crucialcor} implies $g=\1_\vv$. 
\end{proof}

\begin{corollary}\label{corsurj}  Let $\w,\vv\in \N^\calV$ arbitrary and either $\char(\K)=0$ or $\char(\K)>\sum_{i\in \calV} \vv_i.$  Then the derivative \beq \label{surj} T_{(v,w)}\mu_{\vv,\w}: T_{(v,w)} \M_{\vv,\w}\to T_{\mu_{\vv,\w}(v,w)} \g^*_\vv\eeq is surjective if and only 
$\overline{\varrho}_{\vv,\w}(x)(v,w)=0$ implies $x=0$. In particular  at a solution $(v,w)$ of \eqref{solnlambda} the derivative 
\eqref{surj} is surjective.\end{corollary} 
\begin{proof} The map $T_{(v,w)}\mu_{\vv,\w}$ is not surjective if and only if
	there exists $0\neq x\in \g_\vv$ such that  for all $Y\in  T_{(v,w)} \M_{\vv,\w}$
	 $$0=T_{(v,w)}\mu_{\vv,\w}(Y)f^x=Y(\mu_{\vv,\w}\circ f^x)=df^x(Y)=\omega(\overline{\varrho}(x)(v,w),Y),$$ which holds when $\overline{\varrho}_{\vv,\w}(x)(v,w)=0$, $\omega$ being non-degenerate.  Here $f^x$ is the function defined
	in \eqref{fx} while the last equation follows from the moment map property 
	\eqref{momentdefn}. This proves the first statement.
	
	The second statement follows from the first and Proposition~\ref{crucialcor}.
\end{proof}

Now we assume $\K$ is algebraically closed and either $\char(\K)=0$ or \begin{eqnarray} \label{charcond}\char(\K)>\sum_{i\in \calV} \vv_i.\end{eqnarray} For $\lambda\in \K$ we take  ${\lambda \1_\vv}=(\lambda Id_{V_1},\dots,\lambda Id_{V_n})\in (\g_{\vv}^*)^{\rmG_{\vv}}$, and define the affine variety $${\mathcal{V}}_{\lambda}(\vv,\w)=\mu^{-1}_{\vv,\w}(\lambda \1_\vv).$$ For $l\in \Z$ we define the character $\chi^l:\rmG_\vv\to \K^\times$ by
$\chi^l(g)=\Pi_{i\in \calV} \det(g_i)^l$. With these we can define $$\K\left[{\mathcal{V}}_{\lambda}(\vv,\w)\right]^{\rmG_\vv,\chi^{l}}:=\left\{ f\in \K\left[{\mathcal{V}}_{\lambda}(\vv,\w)\right] | f(g(x))= \chi^l(g) f(x)  \mbox{ for all } x\in {\mathcal{V}}_{\lambda}(\vv,\w)\right\},$$ so $\bigoplus_{n\in \N} \K\left[{\mathcal{V}}_{\lambda}(\vv,\w)\right]^{\rmG_\vv,\chi^{ln}}$ becomes an $\N$-graded algebra and  so we can define the 
 Nakajima quiver variety \cite{nakajima} as the GIT quotient:
$$\calM_{l,\lambda}(\vv,\w)={\rm Proj}\left(\bigoplus_{n\in \N} \K\left[{\mathcal{V}}_{\lambda}(\vv,\w)\right]^{\rmG_\vv,\chi^{ln}} \right).$$  For now we note that as an affine GIT quotient 
 $$\calM_{0,\lambda}(\vv,\w)={\rm Spec} \left(\K\left[{\mathcal{V}}_{\lambda}(\vv,\w)\right]^{\rmG_\vv}\right)={\mathcal{V}}_{\lambda}(\vv,\w)/\!/\G_v$$ is an affine variety. Exactly as in \cite[Corollary 3.12]{nakajima} we get that $\calM_{1,\lambda}(\vv,\w)$ is non-singular of dimension \beq \label{dimension} 2d_{\vv,\w}=2\left(\sum_{e\in \calE} \vv_{s(e)}\vv_{t(e)}+\sum_{i\in \calV} \vv_i(\w_i-\vv_i)\right)\eeq for all $\lambda$. We also have 
\begin{lemma} For $\lambda \neq 0$ the variety $\calM_{0,\lambda}(\vv,\w)$
 is non-singular of dimension $2d_{\vv,\w}$ and hence $\calM_{1,\lambda}\cong \calM_{0,\lambda}$. \end{lemma}
\begin{proof} First we note that by Corollary~\ref{corsurj} at a solution of 
	\eqref{solnlambda} the derivative \eqref{surj} surjective. This shows that ${\mathcal{V}}_{\lambda}(\vv,\w)=\mu^{-1}_{\vv,\w}(\lambda\1_\vv)$ is non-singular of dimension $\dim(\M_{\vv,\w})-\dim(\g^*_\vv)$. Now by Corollary~\ref{simple} the action of $\G_\vv$ on ${\mathcal{V}}_\lambda({\vv,\w})$ is free, and 
	therefore the quotient $\calM_{0,\lambda}={\cal V}_\lambda({\vv,\w})/\!/\rmG_\vv$ is non-singular of dimension $\dim({\mathcal{V}}_{\lambda}(\vv,\w)) -\dim( G_\vv)=2d_{\vv,\w}$. 
	
	Finally by the GIT construction we have the map $\calM_{1,\lambda}(\vv,\w)\to \calM_{0,\lambda}$, which is proper and a resolution of singularities. Consequently it is an isomorphism. 
 	\end{proof}

 \begin{theorem} \label{pure} When $\K=\C$ the mixed Hodge structure on the isomorphic cohomologies $H^*(\calM_{1,0}(\vv,\w))\cong H^*(\calM_{0,1}(\vv,\w))$ is pure.
 \end{theorem}
 \begin{proof}  The proof is similar to \cite[Proposition 2.2.6]{hausel-etal}. We define 
	the map $\mu:\M_{\vv,\w}\times \C \to \g^*_\vv$ by $$\mu(v,w,z)=\mu_{\vv,\w}(v,w)-z \1_\vv^*.$$  We then take ${\mathcal{V}}:=\mu^{-1}{0}$ and define: 
	$$\calM_1:={\rm Proj}\left(\bigoplus_{n\in \N} \C\left[{\mathcal{V}}\right]^{\rmG_\vv,\chi^{n}} \right),$$ where we extended the $\rmG_\vv$ action trivially over $\C$. The forgetful map $f(v,w,z)=z$
	induces a map $f:\calM_1\to \C$. Now, similarly to \cite[Corollary 3.12]{nakajima},
	one can prove that $\calM_1$ is non-singular. Also $f$ is a submersion as 
	$\partial_z \mu=1_\vv$. Finally we have the  $\C^\times$-action \beq \label{ctimes}\lambda(v,w,z)=(\lambda v,\lambda w,\lambda^2 z)\eeq descending to a $\C^\times$-action on $\calM_1$. We have the natural $\C^\times$-equivariant proper map $\calM_1\to \calM_0$ where $\calM_0:\rm{Spec}[{\mathcal{V}}]^{\rmG_\vv}$ is the affine GIT quotient, with the $\C^\times$-action given by \eqref{ctimes}. It is clear that the fixed point set $(\M_{\vv,\w}\times \C)^{\C^\times}=\{\0\}$ consists of the origin, and also that every $\C^\times$-orbit on $\M_{\vv,\w}\times \C$ will have the origin $\0$ in its closure. It follows that the $\C^\times$ action on $\calM_0$ will have a single fixed point - the image of $\0$ by the quotient map-  and all $\C^\times$-orbits on $\calM_0$ will have
	this single fixed point $\0$ in their closure. It follows that $\calM_1^{\C^\times}$ is proper and for any $x\in \calM_1$ the $\lim_{\lambda\to 0} z\lambda$ exists. The statement now follows from \cite[Appendix B]{hausel-etal}. 
	 \end{proof}
	
Now we can determine the Betti numbers of quiver varieties.
\paragraph{\it Proof of Theorem~\ref{maint}.} The strategy of the proof is the following (cf. \cite[\S 2.1, Appendix]{hausel-villegas}, \cite[\S 2.5]{hausel-etal}). 
First we construct a spreading out of $\calM_{0,1}(\vv,\w)/\C$. In other words we construct a  scheme   $\Y$ over a finitely generated $\Z$-algebra $R$  together with a homomorphism $\varphi:R\to \C$, such that the extension of scalars gives $\Y_\varphi \cong \calM_{0,1}(\vv,\w)/\C$. Then for any  homomorphism $\varphi:R\to \F_q$ to finite fields we will use Proposition~\ref{main} to count the number of rational
points $\Y_\varphi(\F_q)$ and find it is a polynomial $E(q)$ in $q$. Then Katz's \cite[Theorem 6.1.2.3]{hausel-villegas} will
imply that the so-called $E$-polynomial satisfies $E(\calM_{0,1}(\vv,\w);x,y)=E(xy)$. Combining  Theorem~\ref{pure} and \cite[Proposition 2.5.2]{hausel-etal}   we can deduce that $$E(q)=\sum_{i} b_{2i}(\calM_{0,1}(\vv,\w)/\C) q^{d_{\vv,\w}-i},$$ which will yield Theorem~\ref{maint}.

So let $D=\sum {\vv_i}$ and let $R:=\Z[\frac{1}{D!}]$. The choice of $R$ will ensure that there is a homomorphism 
$\varphi:R\to \K$ if and only if $\char(\K)$ satisfies \eqref{charcond}. We can now construct everything $\M_{\vv,\w}(R)$ $\G_\vv(R)$ and $\g^*_\vv(R)$ over $R$, so that we will have $\mu_{\vv,\w}:\M_{\vv,\w}(R)\to \g^*_\vv(R)$. The easiest
is to think in terms of the explicit form \eqref{explicit} for $\mu_{\vv,\w}$. Let $\X$ be the closed affine subscheme of $\M_{\vv,\w}(R)$ given by the equations $\mu_{\vv,\w}(v,w)=\1_\vv$. This is clearly a spreading out of 
${\cal V}_1(\vv,\w)/\C$. 
We can now define $\Y=\rm{Spec}(R[\X]^{\G_\vv(R)}).$
As the natural map $\varphi: R\to \C $ is a flat morphism \cite[Lemma 2]{seshadri} implies that $\Y$ is indeed a spreading out of $\calM_{0,1}(\vv,\w)/\C$. Now if $\varphi:R\to \F_q$ a homomorphism then  by Corollary~\ref{simple} $\G_\vv(\F_q)$ acts
freely on ${\mathcal{V}}_{1}(\vv,\w)(\F_q)$ so  we have (cf. \cite[proof of Theorem 3.5.1]{hausel-villegas}) that $$|\Y_\varphi(\F_q)|=\frac{|{\mathcal{V}}_{1}(\vv,\w)(\F_q)|}{|\G_\vv(\F_q)|}.$$

For $\w\in \calV^\N$ we will determine
the grand generating function \beq \label{grandquiver}\Phi(\w)=
\sum_{\vv\in \calV^\N} \frac{|{\mathcal{V}}_{1}(\vv,\w)(\F_q)|}{|\G_\vv(\F_q)|} \frac{|\g_{\vv}|}{|\V_{\vv,\w}|}X^\vv=\sum_{\vv\in \calV^\N}\sum_{x\in \g_\vv} \frac{ a_{\varrho_{\vv,\w}
}(x) \Psi(\rm{\tr_\vv}(x))}{|\G_\vv(\F_q)|}X^\vv=\sum_{\vv\in \calV^\N}\sum_{[x]\in \g_\vv/\G_\vv} \frac{ a_{\varrho_{\vv,\w}
}(x) \Psi(\rm{\tr_\vv}(x))}{|\calC_x|}X^\vv.\eeq Here we first applied our main Proposition~\ref{main} to get $$|{\mathcal{V}}_{1}(\vv,\w)(\F_q)|=|\g_\vv|^{-1} |\V_{\vv,\w}| \sum_{x\in \g_\vv} a_{\varrho_{\vv,\w}
}(x) \Psi(\langle x, \1^*_\vv\rangle),$$  and  then rewrote the sum over $\g_\vv/\G_\vv$  the set of $\G_\vv$ orbits on $\g_\vv$ with $\calC_x
\subset \G_\vv$ standing for the centralizer of $x$. We also used the notation \beq \label{formal}X^\vv=\Pi_{i\in \calV}X_i^{\vv_i}.\eeq To understand now \eqref{grandquiver} we should recall some basic linear algebraic facts for  the collection $x=(x_i)_{i\in \calV}$ of linear endomorphisms. Any linear
endomorphism $x_i\in \gl(V_i)$ decomposes as $ x^{nil}_i\oplus x^{reg}_i$ on the decomposition $V_i=N(x_i)\oplus R(x_i)$, where 
$N(x_i)=\rm{ker}(x^{\vv_i}_i)$ and $R(x_i)=\rm{im}(x^{\vv_i}_i)$. By construction $x^{reg}_i$ will be a non-singular transformation on $R(x_i)$, while $x^{nil}_i$ will be nilpotent on $N(x_i)$. Similarly to \cite[Lemma 1]{feit-fine} it is straightforward to check that if $\phi\in \ker(\varrho_{\vv,\w}(x))\subset \V_{\vv,\w}$ then $\phi$ preserves the decomposition $V_i=N(x_i)\oplus R(x_i)$ and $W_i=W_i\oplus 0$. Consequently, we have $a_{\varrho_{\vv,\w}}(x)=a_{\varrho_{\vv^{nil},\w}}(x^{nil})a_{\varrho_{\vv^{reg},0}}(x^{reg}),$ similarly $\Psi(\tr_\vv(x))=\Psi(\tr_{\vv^{nil}}(x^{nil}))\Psi(\tr_{\vv^{reg}}(x^{reg}))=\Psi(\tr_{\vv^{reg}}(x^{reg}))$ as $
\tr_{\vv^{nil}}(x^{nil})=0,$ and clearly $|\calC_{x}|= |\calC_{x^{nil}}||\calC_{x^{reg}}|$. Finally observe that if     $x=x^{nil}$ is nilpotent $\Psi(\tr_\vv(x))=1$.

Now we can factor \eqref{grandquiver} as \beq \label{functional} 
\Phi(\w)=\Phi_{nil}(\w) \Phi_{reg}.\eeq 
Here  we used,  for a dimension vector $\w\in \calV^\N$, the 
generating functions 
$$\Phi_{nil}(\w)=
\sum_{\vv\in \calV^\N}\sum_{[x]\in (\g_\vv/\G_\vv)^{nil}} \frac{ a_{\varrho_{\vv,\w}
}(x)}{|\calC_x|}X^\vv 
$$ and
$$\Phi_{reg} = \sum_{\vv\in \calV^\N}\sum_{[x]\in (\g_\vv/\G_\vv)^{reg}} \frac{ a_{\varrho_{\vv,\0}
}(x) \Psi(\rm{\tr_\vv}(x))}{|\calC_x|}X^\vv. 
$$

 Finally we note, that when $\w=\0$ Lemma~\ref{lemma} implies that $\Phi(\0)=1$ and so \eqref{functional} yields  $\Phi_{reg}=\frac{1}{\Phi_{nil}(\0)}$, giving the result \beq \label{quotient}\Phi(\w)=\frac{\Phi_{nil}(\w)}{\Phi_{nil}(\0)}.\eeq We are left with  understanding
$\Phi_{nil}(\w)$,  which reduces to some explicit formulae for which we will need some new notation.

 Let $V_1$  and $V_2$ be finite dimensional $\F_q$ vector spaces of dimension $v_1$ and $v_2$ respectively. Let $\lambda^i=(\lambda^i_1,\lambda^i_2,\dots)$ be a partition of $v_i$, i.e.
$v_i=\lambda^i_1+\lambda^i_2+\dots$ and $\lambda^i_1\geq \lambda^i_2\geq \dots$ are positive integers. Let  $x_i\in\gl(V_i)$ be a nilpotent endomorphism of type $\lambda^i$, in
other words a nilpotent endomorphism with Jordan normal form consisting of nilpotent Jordan blocks of size $\lambda^i_j$. If $\calC_{x_i}=\{g\in\GL(V_i)| gx_i=x_ig\}$ denotes the centralizer of $x_i$ then it follows from  \cite[Theorem 3.4]{hua} $$|\calC_{x_i}|=q^{\langle\lambda^i,\lambda^i\rangle}\prod_k \prod_{j=1}^{m_k(\lambda^i)} (1-q^{-j}),$$ where $m_k(\lambda_i)$ is the multiplicity of the part $k$ in the partition $\lambda_i$, i.e. $\lambda_i=(1^{m_1(\lambda^i)},2^{m_2(\lambda^i)},\dots)$ in the multiplicity notation of partitions. Furthermore for any two partitions $\lambda^1=(1^{m_1(\lambda^1)},2^{m_2(\lambda^1)},\dots))$ and $\lambda^2=(1^{m_1(\lambda^2)},2^{m_2(\lambda^2)},\dots))$ we use the notation \beq\label{partitionotation}\langle \lambda^1,\lambda^2\rangle=\sum_{i,j} {\rm{min}}(i,j) m_i(\lambda^1)m_j(\lambda^2).\eeq

Then it follows from \cite[Lemma 3.3]{hua} that $$|\{ A\in {\rm{Hom}}(V_1,V_2) | A X_1=X_2A \}|=q^{\langle \lambda^1,\lambda^2\rangle}.$$ In particular when $\lambda_2=(1^{v_2})$ i.e. $X_2=0$, we get $$|\{ A\in {\rm{Hom}}(V_1,V_2) | A X_1=0 \}|=q^{\langle \lambda^1,1^{v_2}\rangle}.$$ 

These results put together lead to the required combinatorial formula:
$$\Phi_{nil}(\w)=\sum_{\lambda  \in \calP^\calV}  \frac{\left(\prod_{e\in \calE} q^{\langle \lambda^{s(e)},\lambda^{t(e)}\rangle} \right)\left(\prod_{i\in \calV} q^{\langle\lambda^i,1^{\w_i}\rangle}\right)}{\prod_{i\in \calV} \left(q^{\langle\lambda^i,\lambda^i\rangle}\prod_k \prod_{j=1}^{m_k(\lambda^i)} (1-q^{-j})\right) }X^{|\lambda|}.$$ Via \eqref{quotient} we obtain a formula for the generating function 
for $|{\cal V}_1(\vv,\w)(\F_q)|$. This shows that $|{\cal V}_1(\vv,\w)(\F_q)|$ is
a rational function in $q$, but as it is an integer for any prime power $q$ (satisfying \eqref{charcond}) it has to be a polynomial in $q$. Therefore Katz's 
\cite[Theorem 6.1.2.3]{hausel-villegas} applies and yields Theorem~\ref{main}. $\square$

\section{Solution of Kac's conjecture}
 A representation of $\Gamma$  over a field  $\K$ is a collection of finite dimensional $\K$-vector spaces $\{V_i\}_{i\in \calV}$
and linear maps $\{\varphi_e\}_{e\in \calE}$ such that $\varphi_e\in \Hom_\K(V_{s(e)},V_{t(e)})$. The dimension vector $\alpha\in \N^{\calV}$  encodes the dimensions $\alpha_i=\dim_\K (V_i)$. 
The building blocks of representations of $\Gamma$ are the {\em indecomposable} ones, that is ones that cannot be written as a non-trivial direct sum of two representations. Kac proved  in \cite[Theorems 1,2]{kac1} that there exists an indecomposable representation of $\Gamma$ of dimension vector $\alpha$ if and only if $\alpha$ is a root of a certain Kac-Moody Lie algebra $\g(\Gamma)$ associated to $\Gamma$. 

When the field $\K$ is not algebraically closed, like the finite fields $\F_q$ we will focus in this paper, then it is natural to study {\em absolutely indecomposable representations} of $\Gamma$, that is ones that are indecomposable over
the algebraic closure $\overline{\K}$. Kac in \cite{kac2} introduced the number $A_\Gamma(\alpha,q)$ of indecomposable representations of $\Gamma$ over the finite field $\K=\F_q$ of dimension vector $\alpha$.  He proved \cite[Proposition 1.5]{kac2} that $A_\Gamma(\alpha,q)$ is a polynomial in $q$ with integer coefficients. He went on to conjecture  \cite[Conjecture 1]{kac2}
that the constant term of this polynomial $A_\Gamma(\alpha,q)$ captures more information
on the associated Kac-Moody Lie algebra:

\begin{conjecture}[Kac] \label{kacconj}The constant term $$A_\Gamma(\alpha,0)=m_\alpha$$ equals the multiplicity of the weight $\alpha$ in $\g(\Gamma)$. 
\end{conjecture}

 The Kac-Moody algebra \cite{kac0} is constructed from the Cartan matrix  of the quiver as an analouge of a simply-laced simple Lie algebra, which arise from the $ADE$ Dynkin diagrams. 
 Instead of giving details of that construction we give a combinatorial definition of 
the numbers $m_\alpha$. Let $(\alpha_i)_{i\in \calV}$ be the standard
basis of $\Z^{\calV}$. Define a bilinear pairing on $\Z^\calV$  by setting $(\alpha_i,\alpha_j)=\delta_{ij} - \frac{b_{ij}}{2}$, where $b_{ij}$ is the number of 
edges of $\Gamma$ between $i$ and $j$.  One defines the reflections $r_i:\Z^\calV\to \Z^\calV$ by $r_i(\lambda)=\lambda-2(\lambda,\alpha_i)\alpha_i$. The subgroup  $W=\langle (r_i)_{i\in \calV} \rangle \subset {\rm Aut}(\Z^\calV)$ generated by these reflections is the Weyl group. Extend now the action of $W$ to the lattice
$\Z^{\calV}\oplus \Z\rho$ by $r_i(\rho)=\rho-\alpha_i$. For any $w\in W$ 
we have $\rho - w(\rho)\in \N^\calV\setminus 
\{ {\mathbf 0} \}$.
So if for an $\alpha=\sum_i k_i \alpha_i\in \Z^{\calV}$ we write $X^\alpha=\prod_{i\in \calV} X_i^{k_i}$ then we can write the product expansion of the following sum \beq \label{kac}\sum_{w\in W} \det(w) X^{\rho-w(\rho)} = \prod_{\alpha \in \N^\V} (1-X^\alpha)^{m_\alpha},\eeq defining the integers $m_\alpha$. When $m_\alpha$ is defined through the usual route as the multiplicity of the weight $\alpha$ in $\g(\Gamma)$, \eqref{kac} is called the Kac denominator formula. 

There is a similar generating formula encoding the coefficients
of the $A$-polynomial due to Hua \cite{hua}. If $A_\Gamma(\alpha,q)=\sum_j t^\alpha_j q^j$ then we have \cite[Theorem 4.9]{hua}

\begin{theorem}[Hua] If for a
collection
of partitions $\lambda=(\lambda_i)_{i\in \calV} \in \P^{\calV}$ we denote the vector  $|\lambda|=(|\lambda_i|)_{i\in \calV}\in \N^{\calV}$  of the sizes of the partitions, then 
\beq \label{hua} \sum_{\lambda  \in \calP^\calV}   \frac{\prod_{e\in \calE} q^{\langle \lambda^{s(e)},\lambda^{t(e)}\rangle}}{\prod_{i\in \calV} \left(q^{\langle\lambda^i,\lambda^i\rangle}\prod_k \prod_{j=1}^{m_k(\lambda^i)} (1-q^{-j})\right) }X^{|\lambda|} = \prod_{\alpha\in \N^{\calV}} \prod_{i=0}^\infty \prod_{j=0}^\infty (1-q^{i+j} X^\alpha)^{t^\alpha_j}  \eeq \end{theorem}

As Hua points out in \cite{hua}[Corollary 4.10] if one could show that
the $q=0$ evaluation of the combinatorially defined LHS of \eqref{hua} 
would agree with the combinatorially defined LHS of \eqref{kac} then one would have that  \beq \label{agrees}  \sum_{\lambda\in \calP^{\calV}}\prod_{\alpha \in \N^\V} (1-X^\alpha)^{m_\alpha} = \prod_{\alpha\in \N^{\calV}}(1- X^\alpha)^{t^\alpha_0} \eeq
the $q=0$ evaluation of the RHS of \eqref{hua} agrees with the RHS of \eqref{kac}. In turn, by comparing coefficients, \eqref{agrees} implies Kac's Conjecture~\ref{kacconj} that $m_\alpha=t^\alpha_0$. 
However the combinatorial expressions on the RHS of \eqref{kac} and \eqref{hua} are difficult to evaluate in practice \cite{sevenhant-vandenbergh,mozgovoy}. 

In this paper we take an alternative route  to prove \eqref{agrees} and thus Conjecture~\ref{kacconj} by relating the RHS of 
\eqref{kac} and \eqref{hua} to the denominator of the Weyl-Kac character formula.
Recall the notation from \cite[\S 2]{nakajima}. Thus $P$ denotes a finitely generated free $\Z$-module, the {\em weight lattice}. 
Let $P^*=\Hom_\Z(P,\Z)$ with natural pairing $\langle,\rangle:P\times P^*\to \Z$. Now $\alpha_i\in P$ (i=1..n) are the linearly
independent {\em simple roots}. Let $Q=\Z^{\calV}=\Z[\alpha_1,\dots,\alpha_n]$ be the {\em root lattice}, we have encountered above. Have a symmetric bilinear form $(,)$ on $P$ such that
the {\em coroots} $h_i\in P^*$ 
given by $\langle h_i,\lambda\rangle =(\alpha_k,\lambda)$ satisfy that $\langle h_i,\alpha_j\rangle=\delta_{ij}-\frac{b_{ij}}{2}$. Thus the root
lattice $Q=\Z^\calV$ will inherit the pairing we already introduced above. 
Finally we pick $\Lambda_i\in P$ for $i=1,\dots,n$ such that $\langle h_j,\Lambda_i\rangle =\delta_{ij}$. For $\w\in \N^\calV$ we introduce $\Lambda_\w=\sum_i w_i\Lambda_i$. Finally we can define reflections $r_i:P\to P$ given by 
$r_i(\lambda)=\lambda-(\lambda,\alpha_i)\alpha_i$, which then generate the Weyl group $W$. It is clear that
the root lattice $Q$ is left invariant and the induced action of $W$ is the one we encountered above. Note also that $\Lambda_{\1}=\sum_i \Lambda_i$
plays the role of $\rho$ introduced above.

Now the Weyl-Kac character  formula   \cite[Theorem 10.4]{kac0} says the following:

\begin{theorem}[Kac]  Let  $L(\Lambda)$ be an irreducible representation of $\g(\Gamma)$ of highest weight  $\Lambda\in P$.  Let $L(\Lambda)=\oplus_{\alpha\in \N^{\calV}} L(\Lambda)_{\Lambda-\alpha}$ denote its weight space
decomposition. Then \beq \label{weyl-kac} \sum_{\alpha\in \N^{\calV}} \dim\left(L(\Lambda)_{\Lambda-\alpha}\right) X^{\alpha}=\frac{{\displaystyle \sum_{w\in W}} \det(w) X^{\Lambda+\rho-w(\Lambda+\rho)}}{{\displaystyle \sum_{w\in W} }\det(w) X^{\rho-w(\rho)}}\eeq 
\end{theorem}

Nakajima in \cite{nakajima} gives a geometrical interpretation of the irreducible representation $L(\Lambda)$, and in turn to the LHS of 
\eqref{weyl-kac} using his quiver varieties.  In the special case of $x=0$
{\cite[Theorem 10.2]{nakajima}} implies: 
\begin{theorem}[Nakajima] Fix $\w\in \N^{\calV}$ then
 \beq \label{nakajima}  \sum_{\vv\in \N^{\calV}} \dim\left(H^{2d_{\vv,\w}}(\calM(\vv,\w)\right) X^{\vv}= {\sum_{\alpha\in \N^{\calV}}} \dim\left(L(\Lambda_\w)_{\Lambda_\w-\alpha}\right) X^{\alpha}
 \eeq
\end{theorem}

We notice that the denominator of the RHS of \eqref{hausel} agrees
with the LHS of Hua's \eqref{hua}. On the other hand by Theorem~\ref{maint} and Nakajima's \eqref{nakajima} we have   \beq \label{chain} (\mbox{LHS of \eqref{hausel}})_{q=0}= \mbox{ LHS of \eqref{nakajima} } = \mbox{LHS of \eqref{weyl-kac}},  \eeq for every $\Lambda_\w=\Lambda$. For $m\in \N$ take in particular $\w=m{\bf 1}$ then $\Lambda_{m{\bf 1}}=m\rho$. Then by applying the Kac denominator formula \eqref{kac} for both the numerator and the denominator we get $$(\mbox{RHS of \eqref{weyl-kac}})_{\Lambda=m\rho}= \frac{{\displaystyle \sum_{w\in W}} \det(w) X^{(m+1)\rho-w((m+1)\rho)}}{{\displaystyle \sum_{w\in W} }\det(w) X^{\rho-w(\rho)}}=\frac{\prod_{\alpha \in \N^\V} (1-X^{(m+1)\alpha})^{m_\alpha}}{\prod_{\alpha \in \N^\V} (1-X^\alpha)^{m_\alpha}}=\prod_{\alpha \in \N^\V} (1+X^{\alpha}+\dots+X^{m\alpha})^{m_\alpha}.$$ Thus when we  take the limit $m\rightarrow \infty$ we get 
\beq \label{limitweyl-kac}\lim_{m\rightarrow \infty}\left( (\mbox{RHS of \eqref{weyl-kac}})_{\Lambda=m\rho} \right) = \frac{1}{\prod_{\alpha \in \N^\V} (1-X^\alpha)^{m_\alpha}}. \eeq  Because for a partition $\lambda^i$ of $n>0$ we have $\langle \lambda_i,1^m\rangle\geq m$ and so in the $m\rightarrow \infty$ limit the numerator of the RHS of \eqref{hausel} tends to $1$ therefore
\beq \label{limithausel} \lim_{m\rightarrow \infty}\left( (\mbox{RHS of \eqref{hausel}})_{\w=m\1 } \right) = \frac{1}{ \displaystyle{\sum_{\lambda  \in \calP^\calV}}   \frac{\prod_{e\in \calE} q^{\langle \lambda^{s(e)},\lambda^{t(e)}\rangle}}{\prod_{i\in \calV} \left(q^{\langle\lambda^i,\lambda^i\rangle}\prod_k \prod_{j=1}^{m_k(\lambda^i)} (1-q^{-j})\right) }X^{|\lambda|}}. \eeq
Now the combination of \eqref{limithausel}, \eqref{hua}, \eqref{chain}, \eqref{weyl-kac}, \eqref{limitweyl-kac} yields \eqref{agrees}, which, as explained there, implies Conjecture~\ref{kacconj}.$\square$

\end{document}